\documentclass[a4paper, 12pt,reqno]{amsart}

\usepackage{amsmath}
\usepackage{amssymb}
\usepackage[dvipsnames]{xcolor}
\usepackage{hyperref} 
\usepackage{amsfonts}
\usepackage{mathtools}
\usepackage{theoremref}
\usepackage{amsthm}
\usepackage[dvipsnames]{xcolor}
\usepackage[utf8]{inputenc}
\usepackage{array}
\usepackage[inline]{enumitem}
\usepackage[english]{babel}
\usepackage{comment}
\usepackage{multicol}

\usepackage{tikz}
\usetikzlibrary{shapes.geometric}
\usetikzlibrary {arrows.meta}

\usetikzlibrary{hobby}
\usetikzlibrary{positioning}

\calclayout

\setlist{topsep=0mm,partopsep=0mm,itemsep=1mm}

\theoremstyle{plain}
\newtheorem{thm}{Theorem}[section]

\newtheorem{lemma}[thm]{Lemma}

\theoremstyle{definition}

\newtheorem{remark}[thm]{Remark}

\newtheorem*{rep@theorem}{\rep@title}
\newcommand{\newreptheorem}[2]{%
\newenvironment{rep#1}[1]{%
 \def\rep@title{#2 ##1}%
 \begin{rep@theorem}}%
 {\end{rep@theorem}}}
\makeatother

\newreptheorem{theorem}{Theorem}
\newreptheorem{corollary}{Corollary}

\theoremstyle{remark}

\newtheorem{cclaim}{Claim}[thm]

\newcommand{\N}{\mathbb{N}}
\newcommand{\Z}{\mathbb{Z}}
\newcommand{\Q}{\mathbb{Q}}

\newcommand{\inlineitem}[1][]{%
\ifnum\enit@type=\tw@
    {\descriptionlabel{#1}}
  \hspace{\labelsep}%
\else
  \ifnum\enit@type=\z@
       \refstepcounter{\@listctr}\fi
    \quad\@itemlabel\hspace{\labelsep}%
\fi}

\newcommand{\revision}[1]{{\color{black} #1}}

\newcommand{\RM}{\mathcal{M}[G;I,\Lambda;P]}
\newcommand{\Cong}{\text{Cong}}
\newcommand{\Diag}{\text{Diag}}
\newcommand{\normspace}[1]{\, #1 \,}
\newcommand{\bigspace}[1]{\hspace{4.5pt} #1 \hspace{4.5pt}}

\title[Numbers of Diagonal subsemigroups and congruences]{Comparing numbers of diagonal subsemigroups and congruences for semigroups}

\author[C.\ Barber]{Callum Barber}
\address{School of Mathematics and Statistics, University of St Andrews, St Andrews, Fife KY16 9SS, UK.}
\email{cjb38@st-andrews.ac.uk}

\author[N.\ Ru\v{s}kuc]{Nik Ru\v{s}kuc}
\address{School of Mathematics and Statistics, University of St Andrews, St Andrews, Fife KY16 9SS, UK.}
\email{nik.ruskuc@st-andrews.ac.uk}



\begin{document}

\begin{abstract}
    Given a semigroup $S$ a diagonal subsemigroup $\rho$ is defined to be a reflexive and compatible relation on $S$.
    When $S$ is finite, we define the DSC coefficient $\chi(S)$ to be the ratio of the number of congruences to the number of diagonal subsemigroups.
    We have previously observed that $\chi(S) = 1$ if and only if $S$ is a group.
    In this paper we will show that for any rational $\alpha$ with $0 < \alpha \leq 1$, there exists a semigroup with $\chi(S) = \alpha$.
    We will do this by utilizing the Rees matrix construction and adapting the congruence classification of such semigroups to describe the diagonal subsemigroups of it.
\end{abstract}

\maketitle

\section{Introduction}

One of the most fundamental concepts in algebra is that of a congruence.
\revision{Given an algebra $A$, a \text{congruence} on $A$ is an equivalence relation that is compatible with the operations of the algebra.
It is easy to see that every congruence is a subalgebra of the direct square.
A more general type of subalgebra is a \textit{diagonal subalgebra}, which is a reflexive relation that is compatible with the operations of the algebra.
Our previous paper on this topic \cite{BARBER_2025}, explored the relationship between these two objects for different algebraic structures.
For algebras with a Mal'cev term, and in particular for groups in signature $(\cdot,^{-1},1)$, congruences and diagonal subalgebras coincide.
Moving away from Mal'cev algebras though, it is noted that
the same result does not hold in general for semigroups.
We say a semigroup is \emph{DSC} if every diagonal subsemigroup is a congruence, and 
\cite{BARBER_2025} investigates DSC semigroups in some depth.
One such result is the following observation:}

\begin{thm}[{\cite[Corollary 2.5]{BARBER_2025}}]
    \label{motivation theorem}
    Let $S$ be a finite semigroup. Then $S$ is DSC  if and only if $S$ is a group.
\end{thm}

\revision{In this paper we continue this investigation, and provide some 'numerical' indication as to `how badly' finite semigroups fail to be DSC. As such, from this point on we deal exclusively with semigroups, in signature $(\cdot)$, and take groups to mean a special kind of semigroups.
%When we say group here, we mean a semigroup with special properties not an algebra with signature $(\cdot,^{-1},1)$.
%This will also be the case for the rest of the paper.
%\color{black}

It is worth noting in passing that the situation is more subtle in the infinite case.
%For infinite groups things become more interesting.
On the one hand every periodic group 
%(a group where every element has finite order) 
is DSC, and on the other  groups with elements of infinite order are not necessarily DSC, one example being the additive group $(\Z,+)$ of integers.
It is currently unknown whether there exists a DSC group that contains an element of infinite order.
Furthermore, in \cite{BARBER_2025} it is shown that there exist DSC semigroups that are not groups.

Returning to the finite case, we
}
can re-phrase Theorem \ref{motivation theorem} as follows.
Let $\Cong(S)$ be the set of congruences on $S$ and $\Diag(S)$ the set of diagonal subsemigroups on $S$.
When $S$ is finite both of these sets are finite. We define the \textit{DSC coefficient} of $S$ to be:
\begin{equation*}
    \chi(S)  =\frac{|\Cong(S)|}{|\Diag(S)|}.
\end{equation*}
Clearly $\chi(S) \in \Q \cap (0,1]$.
For a finite semigroup $S$:
\begin{equation*}
    \chi(S) = 1 \iff S \text{ is a group}.
\end{equation*}
One can think of the DSC coefficient as a measure of how close to being DSC a semigroup is and wonder what values $\chi(S)$ takes as $S$ ranges over all finite semigroups.
The main aim of this paper is too prove the following result:
\begin{thm}
    \label{Main Theorem}
    For any $\alpha \in \Q \cap (0,1]$ there is a finite semigroup $S$ with $\chi(S) = \alpha$.
\end{thm}

To prove this result we will use the Rees matrix construction.
This construction was originally introduced by Suschkewitsch who employed it to give a full structural description of completely simple semigroups \cite{Su28}.
Here we take a group, $G$, two index sets $I$ and $\Lambda$ and a $\Lambda \times I$ matrix $P$ with entries from $G$.
The Rees matrix semigroup $\RM$ is the set $I \times G \times \Lambda$ with multiplication:
\begin{equation*}
    (i,g,\lambda)(j,h,\mu) = (i , g p_{\lambda j} h ,\mu).
\end{equation*}

The main tool we will use to prove Theorem \ref{Main Theorem} is characterizing the diagonal subsemigroups on a Rees matrix semigroup.
There is a description of the congruences on a Rees matrix semigroup in terms of normal subgroups of the group and equivalence relations on the index sets, originally due to Preston \cite{Preston_1961}.
It turns out that there is an analogous description for the diagonal subsemigroups on a Rees matrix semigroup, namely we can describe them using normal subgroups of the group and reflexive relations on the index sets, with the caveat that the group must be DSC. 

The paper is structured as follows. In section 2 we give the characterization of the diagonal subsemigroups on a Rees matrix semigroup.
In Section 3 we use this characterization to evaluate the DSC coefficient for a Rees matrix semigroup, and show that we can obtain any rational number between zero and one as the DSC coefficient of a Rees matrix semigroup.
In the final section we make some observations about the DSC coefficient of Clifford semigroups and contrast this to the result for Rees matrix semigroups.

We will only require some basic ideas from semigroup theory which will be introduced as they are needed.
For this paper we will use $\N$ to denote the set $\{1,2,3\dots \}$ and $\N_0$ for the set $\{0,1,2,3, \dots\}$.
Given a set $X$ we will let $\Delta_X = \{(x,x) : x \in X \}$ be the diagonal relation on $X$. 
For a subset $\rho$ of $X \times Y$, we will write $\rho^{-1}$ for the set $\{(y,x): (x,y) \in \rho\}$.

\section{Diagonal Subsemigroups on Rees matrix semigroups}

Let $S$ be a Rees matrix semigroup $\RM$ over the group $G$.
We start this section by reviewing the congruence description for Rees matrix semigroups; for a more detailed exploration see \cite[Section 3.5]{Ho95}.
For any $i,j \in I$ and $\lambda , \mu \in \Lambda$ we define an \textit{extract} of the matrix $P$ to be $q_{\lambda \mu i j} = p_{\lambda i} p^{-1}_{\mu i} p_{\mu j} p_{\lambda j}^{-1}$.
We will be interested in triples of the form $(N,\mathcal{S},\mathcal{T})$, where $N$ is a normal subgroup of $G$, $\mathcal{S}$ is a relation on $I$ and $\mathcal{T}$ is a relation on $\Lambda$.
A triple is \textit{linked} if for all $i,j \in I$ and $\lambda,\mu \in \Lambda$:
\begin{align*}
    (i,j) \in \mathcal{S} & \implies q_{\lambda \mu i j} \in N, \\
    (\lambda,\mu) \in \mathcal{T} & \implies q_{\lambda \mu i j} \in N.
\end{align*}
And a triple is an \textit{equivalence triple} if the relations $\mathcal{S}$, $\mathcal{T}$ are equivalence relations.
Often in the literature these linked equivalence triples are just called \textit{linked triples}, but as we will also consider triples where the relations are just reflexive we will use the term linked equivalence triple to distinguish between the two.

If we have a relation $\rho$ on $S$, then we can get a triple $(N_\rho , \rho_I , \rho_\Lambda)$ where:
\begin{align}
    \label{lab1}
    N_\rho & = \{ g \in G : (i,g,\lambda) \normspace{\rho} (i,1,\lambda), \forall i \in I, \lambda \in \Lambda \}, \\
    \label{lab2}
    \rho_I & = \{ (i,j) : (i,p_{\lambda i}^{-1},\lambda) \normspace{\rho} (j,p_{\lambda j}^{-1} , \lambda), \forall \lambda \in \Lambda \}, \\
    \label{lab3}
    \rho_\Lambda & = \{(\lambda,\mu) : (i,p_{\lambda i}^{-1},\lambda) \normspace{\rho} (i,p_{\mu i}^{-1} , \mu) , \forall i \in I \}.
\end{align}
If $(N,\mathcal{S},\mathcal{T})$ is a triple, then we obtain the relation $\rho_{N,\mathcal{S,T}}$ defined by:
\begin{equation}
    \begin{split}
        \label{lab4}
        (i,g,\lambda) \normspace{\rho_{N,\mathcal{S,T}}} (j,h,\mu) \iff & (i,j) \in \mathcal{S},(\lambda , \mu) \in \mathcal{T},\\
        & p_{\nu i} g p_{\lambda k}p_{\mu k}^{-1} h^{-1} p_{\nu j}^{-1} \in N, \forall k \in I , \nu \in \Lambda.
    \end{split}
\end{equation}
Then we have the following result due to Preston \cite{Preston_1961} that gives a direct correspondence between the congruences on a Rees matrix semigroup and the collection of linked equivalence triples:
\begin{thm}
    \label{Cong Classification}
    If $\rho$ is a congruence then $(N_\rho,\rho_I,\rho_\Lambda)$ is a linked equivalence triple and if $(N,\mathcal{S},\mathcal{T})$ is a linked equivalence triple then $\rho_{N,\mathcal{S},\mathcal{T}}$ is a congruence.
    Furthermore the maps $\rho \mapsto (N_\rho,\rho_I,\rho_\Lambda)$ and $(N,\mathcal{S},\mathcal{T}) \mapsto \rho_{N,\mathcal{S},\mathcal{T}}$ are mutual inverses. \qed
\end{thm}

This bijection also preserves the structure of the congruence lattice, as shown by Kapp and Schneider \cite{kapp1969completely}, but it will not be required for this paper. 

A very similar result holds for the diagonal subsemigroups on a Rees matrix semigroup where the group is DSC.
We will show there is a direct correspondence between diagonal subsemigroups on a Rees matrix semigroup and triples of normal subgroups on the group and reflexive relations on the sets $I$ and $\Lambda$.
For the rest of this section let $S = \RM$ where $G$ is a DSC semigroup.
We define a \textit{reflexive triple} to be a triple $(N , \mathcal{S} , \mathcal{T})$ where the relations $\mathcal{S},\mathcal{T}$ are reflexive.

Before we set up a bijection between linked reflexive triples and diagonal subsemigroups we will require the following fact about diagonal subsemigroups on Rees matrix semigroups where the group is DSC.

\begin{lemma}
    \label{Use of DSC}
    Let $\rho$ be a diagonal subsemigroup on $S$.
    If $(i,g,\lambda) \normspace{\rho} (i,h,\lambda)$ then $(i,g^{-1},\lambda) \normspace{\rho} (i,h^{-1},\lambda)$.
\end{lemma}

\begin{proof}
    The $\mathcal{H}$-class $H = H_{i\lambda} = \{ (i,g,\lambda):g\in G\}$ is a subgroup of $\RM$ that is isomorphic to $G$.
    When we restrict $\rho$ to $H$ we get a diagonal subsemigroup on $H$, which must be a congruence as $H$ is DSC.
    So by symmetry, $(i,h,\lambda) \normspace{\rho} (i,g,\lambda)$.
    Therefore:
    \begin{align*}
        (i,g^{-1},\lambda) & = (i,g^{-1} p_{\lambda i}^{-1},\lambda)(i,h,\lambda)(i,p_{\lambda i}^{-1}h^{-1},\lambda) \\
        & \bigspace{\rho} (i,g^{-1} p_{\lambda i}^{-1},\lambda)(i,g,\lambda)(i,p_{\lambda i}^{-1}h^{-1},\lambda) = (i,h^{-1},\lambda). \qedhere
    \end{align*}
\end{proof}

The proofs of the remainder of the results in this section are adapted from \cite[Section 3.5]{Ho95}.

\begin{lemma}
    Let $\rho$ be a diagonal subsemigroup, then $(N_\rho , \rho_I , \rho_\Lambda)$ is a linked reflexive triple.
\end{lemma}

\begin{proof}
    As $\rho$ is reflexive, it is easy to see that $\rho_I$ and $\rho_\Lambda$ are reflexive, and $1 \in N_\rho$.
    Take $g , h \in N_\rho$, then for all $i \in I$ and $\lambda \in \Lambda$ we have $(i,g,\lambda) \normspace{\rho}(i,1,\lambda)$ and $(i,h,\lambda) \normspace{\rho} (i,1,\lambda)$.
    Hence we have:
    \begin{align*}
        (i,gh,\lambda) = (i,g,\lambda) (i,p_{\lambda i}^{-2},\lambda) (i,h,\lambda) \normspace{\rho} (i,1,\lambda) (i,p_{\lambda i}^{-2},\lambda) (i,1,\lambda) = (i,1 , \lambda).
    \end{align*}
    It follows that $gh \in N_\rho$, and so $N_\rho$ is a subsemigroup of $G$.
    As $(i,g,\lambda) \normspace{\rho} (i,1,\lambda)$, by Lemma \ref{Use of DSC} we have that $(i,g^{-1},\lambda) \normspace{\rho} (i,1,\lambda)$ for all $i \in I$ and $\lambda \in \Lambda$.
    Thus $g^{-1} \in N_\rho$ and $N_\rho$ is a subgroup of $G$.
    Finally if we have $n \in G$ then:
    \begin{align*}
        (i,n^{-1}gn,\lambda) & = (i,n^{-1}p_{\lambda i}^{-1},\lambda)(i,g,\lambda)(i,p_{\lambda i}^{-1}n,\lambda) \\
        &\bigspace{\rho} (i,n^{-1}p_{\lambda i}^{-1},\lambda)(i,1,\lambda)(i,p_{\lambda i}^{-1}n,\lambda) = (i,1,\lambda).
    \end{align*}
    Therefore $n^{-1}gn \in N_\rho$ and $N_\rho$ is a normal subgroup of $G$.

    Now let $i,j \in I$ and $\lambda,\mu \in \Lambda$.
    Suppose that $(i,j) \in \rho_I$.
    So $(i,p_{\mu i}^{-1} , \mu) \normspace{\rho} (j,p_{\mu j}^{-1} , \mu)$ and for arbitrary $k \in I$ and $\nu \in \Lambda$ we have:
    \begin{align*}
        (k,p_{\lambda i}p_{\mu i}^{-1} p_{\mu j} p_{\lambda j}^{-1} , \nu) & = (k,1,\lambda)(i,p_{\mu i}^{-1},\mu)(j,p_{\lambda j}^{-1},\nu) \\
        & \bigspace{\rho} (k,1,\lambda)(j,p_{\mu j}^{-1},\mu)(j,p_{\lambda j}^{-1},\nu) \\
        & = (k,1,\nu).
    \end{align*}
    Hence $q_{\lambda \mu i j} = p_{\lambda i}p_{\mu i}^{-1} p_{\mu j} p_{\lambda j}^{-1} \in N_\rho$.
    The case when $(\lambda,\mu) \in \rho_\Lambda$ is analogous, and thus 
    $(N_\rho , \rho_I , \rho_\Lambda)$ is a linked reflexive triple.
\end{proof}

\begin{lemma}
    \label{linked gives diagsubsemi}
    The relation $\rho_{N ,\mathcal{S},\mathcal{T}}$ is a diagonal subsemigroup on $S$.
\end{lemma}

\begin{proof}
    As $\mathcal{S}$ and $\mathcal{T}$ are reflexive it follows that $\rho_{N ,\mathcal{S},\mathcal{T}}$ is reflexive.
    If we have $(i_1,g_1,\lambda_1) \normspace{\rho_{N ,\mathcal{S},\mathcal{T}}} (j_1 , h_1 , \mu_1)$ and $(i_2,g_2,\lambda_2) \normspace{\rho_{N ,\mathcal{S},\mathcal{T}}} (j_2,h_2,\mu_2)$ then we want to show that:
    \begin{equation*}
        (i_1,g_1 p_{\lambda_1 i_2} g_2 , \lambda_2) \normspace{\rho_{N ,\mathcal{S},\mathcal{T}}} (j_1 , h_1 p_{\mu_1 j_2} h_2 , \mu_2).
    \end{equation*}
    This is equivalent to showing:
    \begin{enumerate}
        \item $(i_1,j_1) \in \mathcal{S}$,
        \item $(\lambda_2 , \mu_2) \in \mathcal{T}$,
        \item $Np_{\nu i_1} g_1 p_{\lambda_1 i_2} g_2 p_{\lambda_2 \revision{k}} =Np_{\nu j_1} h_1 p_{\mu_1 j_2} h_2 p_{\mu_2 \revision{k}}, \forall k \in I , \nu \in \Lambda$.
    \end{enumerate}
    Statements (1) and (2) follow immediately from the definition of $\rho_{N ,\mathcal{S},\mathcal{T}}$.
    For (3) fix $k \in I$ and $\nu \in \Lambda$. 
    From the definition of $\rho_{N,\mathcal{S},\mathcal{T}}$ we get that $Np_{\nu i_1}g_1p_{\lambda_1 i_2} =Np_{\nu j_1} h_1 p_{\mu_1 i_2} ,Np_{\mu_1 i_2}g_2p_{\lambda_2 k} =Np_{\mu_1 j_2} h_2 p_{\mu_2 k}$.
    Now we have:
    \begin{align*}
        Np_{\nu i_1} g_1 p_{\lambda_1 i_2} g_2 p_{\lambda_2 k}& = Np_{\nu i_1} g_1 p_{\lambda_1 i_2}Np_{\mu_1 i_2}^{-1} 
        Np_{\mu_1 i_2}g_2p_{\lambda_2 k} \\
        & = Np_{\nu j_1} h_1 p_{\mu_1 i_2} Np_{\mu_1 i_2}^{-1}Np_{\mu_1 j_2} h_2 p_{\mu_2 k}\\
        & = Np_{\nu j_1} h_1 p_{\mu_1 j_2} h_2 p_{\mu_2 k}.\qedhere
    \end{align*}
\end{proof}

\begin{lemma}
    Let $(N , \mathcal{S} , \mathcal{T})$ be a linked reflexive triple with $\rho = \rho_{N ,\mathcal{S},\mathcal{T}}$.
    Then $N_\rho = N , \rho_I  = \mathcal{S}$ and $\rho_\Lambda = \mathcal{T}$.
\end{lemma}

\begin{proof}
    To see that $N_\rho = N$ notice that:
    \begin{align*}
        g \in N_\rho & \iff (i,g,\lambda)\normspace{\rho_{N,\mathcal{S},\mathcal{T}}}(i,1,\lambda),\forall i \in I,\lambda \in \Lambda \\
        & \iff (i,i) \in \mathcal{S} , (\lambda,\lambda) \in \mathcal{T} , Np_{\nu i} g p_{\lambda k}= Np_{\nu i}  p_{\lambda k}, \forall i,k \in I,\lambda,\nu \in \Lambda \\
        & \iff g \in N.
    \end{align*}
    Noting that, as the triple is linked, if $(i,j) \in \mathcal{S}$ then $Np_{\nu i} p_{\lambda i}^{-1} = Np_{\nu j}p_{\lambda j}^{-1}, \forall \lambda , \nu \in \Lambda$.
    So we have:
    \begin{align*}
        (i,j) \in \rho_I & \iff (i,p_{\lambda i}^{-1},\lambda) \normspace{\rho_{N , \mathcal{S} , \mathcal{T}}} (j,p_{\lambda j}^{-1},\lambda), \forall \lambda \in \Lambda \\
        & \iff (i,j) \in \mathcal{S}, (\lambda,\lambda) \in \mathcal{T},Np_{\nu i} p_{\lambda i}^{-1} p_{\lambda k} = Np_{\nu j} p_{\lambda j}^{-1} p_{\lambda k} , \forall k \in I , \lambda,\nu \in \Lambda \\
        & \iff (i,j) \in \mathcal{S} , Np_{\nu i} p_{\lambda i}^{-1} = Np_{\nu j} p_{\lambda j}^{-1}, \forall \lambda, \nu \in \Lambda \\
        & \iff (i,j) \in \mathcal{S}.
    \end{align*}
    Hence $\rho_I = \mathcal{S}$.
    The proof of $\mathcal{T} = \rho_\Lambda$ is dual to this.
\end{proof}

\begin{lemma}
    Let $\rho$ be a diagonal subsemigroup on $S$ with $N = N_\rho , \mathcal{S} = \rho_I$ and $\mathcal{T} = \rho_\Lambda$.
    Then $\rho = \rho_{N,\mathcal{S},\mathcal{T}}$.
\end{lemma}

\begin{proof}
    Let $\rho^\prime = \rho_{N,\mathcal{S},\mathcal{T}}$.
    Suppose $(i,g,\lambda) \normspace{\rho} (j,h,\mu)$,
    to show $(i,g,\lambda) \normspace{\rho^\prime} (j,h,\mu)$ we have to prove that:
    \begin{enumerate}
        \item $(i,j) \in \mathcal{S}$,
        \item $(\lambda,\mu) \in \mathcal{T}$,
        \item $N p_{\nu i} g p_{\lambda k} = N p_{\nu j}hp_{\mu k}, \forall k \in I , \nu \in \Lambda$.
    \end{enumerate}
    Which is in turn equivalent to:
    \begin{enumerate}
        \item $ (i,p_{\nu i}^{-1},\nu) \normspace{\rho} (j,p_{\nu j}^{-1},\nu) ,\forall \nu \in \Lambda$,
        \item $ (k , p_{\lambda k}^{-1},\lambda) \normspace{\rho} (k,p_{\mu k}^{-1},\mu), \forall k \in I $,
        \item $ (l,p_{\nu i} g p_{\lambda k} p_{\mu k}^{-1} h^{-1} p_{\nu j}^{-1},\omega) \normspace{\rho} (l,1,\omega) , \forall k,l \in I , \nu , \omega \in \Lambda$.
    \end{enumerate}
    For (1) let $\nu$ be any element of $\Lambda$ and note that:
    \begin{align*}
        (i,p_{\nu i} g p_{\lambda i}^2,\lambda)& = (i,1,\nu)(i,g,\lambda)(i,p_{\lambda i},\lambda) \\
        & \bigspace{\rho} (i,1,\nu)(j,h,\mu)(i,p_{\lambda i},\lambda) = (i,p_{\nu j} h p_{\mu i} p_{\lambda i} , \lambda).
    \end{align*}
    Let $g^\prime = p_{\nu i} g p_{\lambda i}^2$ and $h^\prime = p_{\nu j} h p_{\mu i} p_{\lambda i}$.
    It follows from Lemma \ref{Use of DSC} that $(i,(g^\prime)^{-1},\lambda) \rho \\ (i,(h^\prime)^{-1},\lambda)$.
    Therefore:
    \begin{equation*}
        (i,g,\lambda)(i,1,\lambda)(i,(g^\prime)^{-1},\lambda)(i,p_{\lambda i}^{-1},\nu) \normspace{\rho} (j,h,\mu)(i,1,\lambda)(i,(h^\prime)^{-1},\lambda)(i,p_{\lambda i}^{-1} , \nu).
    \end{equation*}
    Evaluating both sides gives $(i,p_{\nu i}^{-1},\nu) \normspace{\rho} (j,p_{\nu j}^{-1} , \nu)$.
    The proof of (2) is analogous to this.
    For (3) let $k,l \in I$ and $\nu , \omega \in \Lambda$. Then:
    \begin{align*}
        (l,p_{\nu i} g p_{\lambda k} p_{\mu k}^{-1} h^{-1} p_{\nu j}^{-1} , \omega) & = (l,1,\nu)(i,g,\lambda)(k,p_{\mu k}^{-1} h^{-1} p_{\nu j}^{-1},\omega) \\
        & \bigspace{\rho} (l,1,\nu)(j,h,\mu)(k,p_{\mu k}^{-1} h^{-1} p_{\nu j}^{-1},\omega) \\
        & = (l,1,\omega).
    \end{align*}
    Hence $(i,g,\lambda) \normspace{\rho^\prime} (j,h,\mu)$.
    For the reverse containment suppose $(i,g,\lambda) \normspace{\rho^\prime} (j,h,\mu)$.
    It follows from the definition that of $\rho^\prime$:
    \begin{align*}
        (i,p_{\lambda i} g p_{\lambda i}p_{\mu i}^{-1} h^{-1} p_{\lambda j}^{-1},\lambda) & \normspace{\rho} (i,1,\lambda), \\
        (i,p_{\lambda i}^{-1},\lambda) & \normspace{\rho} (i,p_{\mu i}^{-1},\mu), \\
        (i,p_{\lambda i}^{-1},\lambda) & \normspace{\rho} (j,p_{\lambda j}^{-1},\lambda).
    \end{align*}
    Therefore, the following two quantities are $\rho$-related:
    \begin{align*}
        (i,p_{\lambda i}^{-1},\lambda)(i,p_{\lambda i}^{-2},\lambda)(i,p_{\lambda i} g p_{\lambda i} p_{\mu i}^{-1} h^{-1} p_{\lambda j}^{-1},\lambda)(j,h,\mu)(i,p_{\lambda i}^{-1},\lambda) = (i,g,\lambda),\\
        (j,p_{\lambda j}^{-1},\lambda)(i,p_{\lambda i}^{-2},\lambda)(i,1,\lambda)(j,h,\mu)(i,p_{\mu i}^{-1},\mu) = (j,h,\mu).
    \end{align*}
    Hence $\rho = \rho^\prime$ and this completes the proof.
\end{proof}
By combining the above lemmas we obtain the following analogue of Theorem \ref{Cong Classification}.
\begin{thm}
    \label{diag classification}
    Let $S = \RM$ be a Rees matrix semigroup over a DSC group $G$.
    The maps $\rho \mapsto (N_\rho , \rho_I,\rho_\Lambda)$ and $(N,\mathcal{S},\mathcal{T}) \mapsto \rho_{N,\mathcal{S},\mathcal{T}}$ are mutual inverses and so give a one to one correspondence between diagonal subsemigroups on a Rees matrix semigroup and the set of reflexive linked triples. \qed
\end{thm}

Any finite group is DSC, so we will be able to use the above characterization to give some interesting results on the DSC coefficient for finite Rees matrix semigroups.

\section{DSC Coefficient for Rees Matrix Semigroups}

For this section let $S = \RM$ be a finite Rees matrix semigroup over a group $G$.
Recall the DSC coefficient of $S$ is defined as:
\begin{equation*}
    \chi(S) = \frac{|\Cong(S)|}{|\Diag(S)|}.
\end{equation*}
For $S$ to be finite, the group $G$ must be finite and so is DSC.
Hence Theorems \ref{Cong Classification} and \ref{diag classification} apply, yielding:
\begin{equation}
    \label{form of dsc coeff for RM}
    \chi(\RM) = \frac{\text{number of linked equivalence triples}}{\text{number of linked reflexive triples}}.
\end{equation}

In what follows we will need the following facts about extracts.
\begin{lemma}
    \label{Extract Facts}
    Let $N$ be a normal subgroup of $G$, $i,j,i_0,\dots,i_n \in I$ and $\lambda,\mu,\lambda_0,\dots, \\ \lambda_n \in \Lambda$ be arbitrary. Then the following hold:
    \begin{enumerate}
        \label{Extract Facts 1}
        \item If $i = j$ or $\lambda = \mu$ then $q_{\lambda \mu i j} = 1 \in N$.
        \label{Extract Facts 2}
        \item If $q_{\lambda \mu i j} \in N$, then $q_{\mu \lambda i j} , q_{\mu \lambda j i} , q_{\lambda \mu j i} \in N$.
        \label{Extract Facts 3}
        \item If $q_{\lambda \mu i_0 i_1}, q_{\lambda \mu i_1 i_2}, \dots , q_{\lambda \mu i_{n-1} i_n} \in N$, then $q_{\lambda \mu i_0 i_n} \in N$.
        \label{Extract Facts 4}
        \item If $q_{\lambda_0 \lambda_1 i j},q_{\lambda_1 \lambda_2 i j} ,\dots , q_{\lambda_{n-1} \lambda_n i j} \in N$ then $q_{\lambda_0 \lambda_n i j} \in N$.
    \end{enumerate}
\end{lemma}

\begin{proof}
    \begin{enumerate*}
        \item Obvious. 
        \item Follows from $q_{\lambda \mu i j}^{-1} = q_{\lambda \mu j i}$ and $p_{\mu i} p_{\lambda i}^{-1} q_{\lambda \mu i j} p_{\lambda i} p_{\mu i}^{-1} = q_{\mu \lambda j i}$. 
        \item Follows from $q_{\lambda \mu i_0 i_1} q_{\lambda \mu i_1 i_2}= q_{\lambda \mu i_0 i_2} $. 
        \item Analogous to (3).
    \end{enumerate*}
\end{proof}

From this we obtain the following easily verifiable result.

\begin{lemma}
    \label{Linked equivalence}
    The triple $(N,\mathcal{S},\mathcal{T}) $ is linked if and only if the triples $(N,\mathcal{S},\Delta_\Lambda)$ and $(N,\Delta_I,\mathcal{T})$ are linked. \qed 
\end{lemma}

If we have a relation $\rho$ on a set $X$, we define $\langle \rho \rangle$ to be the intersection of all equivalence relations on $X$ that contain $\rho$.
So $\langle \rho \rangle$ is the smallest (by inclusion) equivalence relation on $X$ that contains $\rho$.
It is well known that $(x,y) \in \langle \rho \rangle$ if and only if there is a sequence $x = x_0 , x_1 , \dots , x_{n-1},x_n = y$, where for each $i = 0 , 1,\dots,n-1$, either $(x_i,x_{i+1}) \in \rho$ or $(x_{i+1},x_i) \in \rho$.
If $\mathcal{R}$ is a set of relations on $X$ then define $\langle \mathcal{R} \rangle = \langle \bigcup_{\rho \in \mathcal{R}} \rho \rangle$.
It turns out that generating an equivalence relation in this way respects the notion of being linked.

\begin{lemma}
    \label{gen implies linked}
    Let $N$ be a normal subgroup of $G$.
    Suppose we have a set of relations $\mathcal{R}_I$ on $I$ such that for each $\mathcal{S} \in \mathcal{R}_I$ the triple $(N,\mathcal{S},\Delta_\Lambda)$ is linked.
    Then $(N,\langle \mathcal{R}_I \rangle,\Delta_\Lambda) $ is a linked triple.
    Dually if $\mathcal{R}_\Lambda$ is a set of relations on $\Lambda$ such that for each $\mathcal{T} \in \mathcal{R}_\lambda$ the triple $(N,\Delta_I,\mathcal{T})$ is linked, then $(N,\Delta_I,\langle \mathcal{R}_\Lambda\rangle)$ is also linked.
\end{lemma}

\begin{proof}
    Let $i,j \in I$ and $\lambda,\mu \in \Lambda$.
    If $(\lambda,\mu) \in \Delta_\Lambda$ then $\lambda = \mu$ and so $q_{\lambda \mu i j} = 1 \in N$.
    Now suppose that $(i,j) \in \langle \mathcal{R}_I \rangle$.
    There is a sequence $i = x_0 , x_1 , \dots , x_n = j$ such that for each $k=0,1, \dots, n-1$, $(x_k,x_{k+1})$ or $(x_{k+1} , x_k) \in \mathcal{S}_k$ for some $\mathcal{S}_k \in \mathcal{R}_I$.
    As the triples $(N,\mathcal{S}_k,\Delta_\Lambda)$ are linked we have that $q_{\lambda \mu x_{k} x_{k+1}}$ or $q_{\lambda \mu x_{k+1} x_k} \in N$.
    It follows from Lemma \ref{Extract Facts} that $q_{\lambda \mu x_0 x_n} = q_{\lambda \mu i j} \in N$.
    Hence the triple $(N,\langle \mathcal{R}_I \rangle,\Delta_\Lambda)$ is linked.
    It follows similarly that $(N,\Delta_I,\langle \mathcal{R}_\Lambda\rangle)$ is also linked.
\end{proof}

As $G$ is finite we can list the normal subgroups of $G$ as $G = N_1 , \dots , N_n = 1$.
Consider the following sets:
\begin{align*}
    E_{I,k} &= \{ \mathcal{S} : (N_k,\mathcal{S},\Delta_\Lambda) \text{ is a linked equivalence triple} \}, \\
    E_{\Lambda,k} &= \{ \mathcal{T} : (N_k,\Delta_I,\mathcal{T}) \text{ is a linked equivalence triple} \}, \\
    R_{I,k} &= \{ \mathcal{S} : (N_k,\mathcal{S},\Delta_\Lambda) \text{ is a linked reflexive triple} \}, \\
    R_{\Lambda,k} &= \{ \mathcal{T} : (N_k,\Delta_I,\mathcal{T}) \text{ is a linked reflexive triple} \}.
\end{align*}

It follows from \eqref{form of dsc coeff for RM} and Lemma \ref{Linked equivalence} that:
\begin{equation}
    \label{DSC coef form for Rees matrix}
    \chi(\RM) = \frac{|E_{I,1}||E_{\Lambda ,1}| + \dots + |E_{I,n}||E_{\Lambda,n}|}{|R_{I,1}||R_{\Lambda ,1}| + \dots +|R_{I,n}||R_{\Lambda,n}|}.
\end{equation}

Thus we will be interested in trying to evaluate the sizes of these sets.
As $N_1 = G$ we have that the triples $(N_1 , \mathcal{S}, \Delta_\Lambda),(N_1,\Delta_I,\mathcal{T})$ are always linked for any $\mathcal{S}$ and $\mathcal{T}$.
So we have that:
\begin{equation}
    \label{N = G case}
    |E_{I,1}| = B(|I|) ,\quad |E_{\Lambda,1}| = B(|\Lambda|),\quad |R_{I,1}| = 2^{|I|^2 - |I|}, \quad|R_{\Lambda,1}| = 2^{|\Lambda|^2 - |\Lambda|}
\end{equation}
and hence:
\begin{equation*}
    \frac{|E_{I,1}||E_{\Lambda,1}|}{|R_{I,1}||R_{\Lambda,1}|} = \frac{B(|I|)B(|\Lambda|)}{2^{|I|^2 - |I|}2^{|\Lambda|^2 - |\Lambda|}}.
\end{equation*}
Here $B(m)$ is the $m^{th}$ bell number, the number of equivalence relations on a set of size $m$, and $2^{m^2 - m}$ is the number of reflexive relations on a set of size $m$.
The number on the right of the equation is the DSC coefficient of the rectangular band $I \times \Lambda$.
It turns out that for fixed $I$ and $\Lambda$ the DSC coefficient of $\RM$ cannot be lower than this.
The main aim of this section is to prove the following result.
\begin{thm}
    \label{main theorem section 3}
    Let $I$ and $\Lambda$ be sets of size $a,b>1$ respectively.
    Then for any group $G$ and $\Lambda \times I$ matrix $P$ with entries from $G$ we have that:
    \begin{equation*}
        \frac{B(a)B(b)}{2^{a^2 - a}2^{b^2-b}} \leq \chi(\RM) < 1.
    \end{equation*}
    Furthermore for any rational $\alpha \in \left[ \frac{B(a)B(b)}{2^{a^2 - a}2^{b^2-b}} , 1\right)$, there is a group $G$ and a matrix $P$ with $\chi(\RM) = \alpha$.
\end{thm}

\begin{remark}
    \label{index sets have size 1 case}
    Let us pause to observe what happens when $I$ or $\Lambda$ has size $1$.
    First assume that $|I| = 1$. For any relation $\mathcal{T}$ on $\Lambda$ and normal subgroup $N$ of $G$ we have that $(N,\Delta_I,\mathcal{T})$ is a linked triple.
    This is because every extract is equal to $1$.
    So as $n$ is the number of normal subgroups of $G$, we have that:
    \begin{equation*}
        \chi(\RM) = \frac{n B(|\Lambda|)}{n 2^{|\Lambda|^2 - |\Lambda|} } = \frac{ B(|\Lambda|)}{ 2^{|\Lambda|^2 - |\Lambda|}}
    \end{equation*}
    An identical result holds when $|\Lambda| = 1$.
\end{remark}

To prove Theorem \ref{main theorem section 3} we will require the following two auxiliary results. The first of which is obvious:
\begin{lemma}
    \label{elementary fraction result}
    Let $a_1 , \dots , a_n , b_1 , \dots , b_n \in \N$, with $\frac{a_1}{b_1} \leq \frac{a_i}{b_i}$ for all $i$.
    Then
    \begin{align*}
        \frac{a_1}{b_1} \leq \frac{a_1 + a_2 + \dots + a_n}{b_1 + b_2 + \dots + b_n}. \tag*{\qed}
    \end{align*}
\end{lemma}

\begin{lemma}
    \label{Bell number inequality}
    Let $n , r , k_1 , \dots , k_r \in \N$ be such that $n = k_1 + \dots + k_r$.
    Then:
    \begin{equation*}
        \frac{B(n)}{2^{n^2 - n}} \leq \frac{B(k_1)}{2^{k_1^2 - k_1}} \frac{B(k_2)}{2^{k_2^2 - k_2}} \cdots \frac{B(k_r)}{2^{k_r^2 - k_r}}.
    \end{equation*}
\end{lemma}

\begin{proof}
    First we will show that:
    \begin{equation*}
        B(s+t) \leq B(s) B(t) 2^{st}.
    \end{equation*}
    Let $X$ be a set of size $s+t$ and let $S,T$ be a partition of $X$ into sets of size $s$ and $t$ respectively.
    Given an equivalence relation $\rho$ on $X$ we obtain equivalence relations $\rho_S = \rho \cap S \times S , \rho_T = \rho \cap T \times T$ on $S$ and $T$ respectively.
    We can also obtain a subset of $S \times T$ from $\rho$:
    \begin{equation*}
        Y_\rho = \{ (s,t): s \in S , t \in T , (s,t) \in \rho \}.
    \end{equation*}
    Clearly $\rho = \rho_S \cup \rho_T \cup Y_\rho \cup Y_\rho^{-1}$, so any equivalence relation on $X$ can be constructed from an equivalence on $S$, an equivalence on $T$ and a subset of $S \times T$.
    There is at most $B(s) B(t) 2^{st}$ such triples, so we have $B(s+t) \leq B(s) B(t) 2^{st}$.
    We will actually only require the weaker result:
    \begin{equation*}
        B(s+t) \leq B(s) B(t) 4^{st}.
    \end{equation*}
    To prove the main result we will use induction on $r$.
    The base case ($r=1$) is clear.
    Suppose that:
    \begin{equation*}
        \frac{B(k_1 + \dots + k_{r-1})}{2^{(k_1 + \dots + k_{r-1})^2 - (k_1 + \dots + k_{r-1})}} \leq \frac{B(k_1)}{2^{k_1^2 - k_1}}  \cdots \frac{B(k_{r-1})}{2^{k_{r-1}^2 - k_{r-1}}}.
    \end{equation*}
    From above we have that $B(n-k_r) B(k_r) \geq B(n) 4^{-k_r(n-k_r)}$, and so:
    \begin{align*}
        \frac{B(k_1)}{2^{k_1^2 - k_1}}  \dots \frac{B(k_{r})}{2^{k_{r}^2 - k_{r}}}  
        & \geq \frac{B(k_1 + \dots + k_{r-1})}{2^{(k_1 + \dots + k_{r-1})^2 - (k_1 + \dots + k_{r-1})}} \frac{B(k_{r})}{2^{k_{r}^2 - k_{r}}} \\
        &= \frac{B(n - k_r)}{2^{(n-k_r)^2 - (n-k_r)}} \frac{B(k_{r})}{2^{k_{r}^2 - k_{r}}} \\
        & \geq \frac{B(n) 4^{-k_r(n-k_r)}}{2^{(n-k_r)^2 - (n-k_r)}2^{k_{r}^2 - k_{r}}} = \frac{B(n)}{2^{n^2 - n}}
    \end{align*}
    as required.
\end{proof}

\begin{proof}[Proof of Theorem \ref{main theorem section 3}.]
    Recall the sets:
    \begin{align*}
    E_{I,k} &= \{ \mathcal{S} : (N_k,\mathcal{S},\Delta_\Lambda) \text{ is a linked equivalence triple} \}, \\
    E_{\Lambda,k} &= \{ \mathcal{T} : (N_k,\Delta_I,\mathcal{T}) \text{ is a linked equivalence triple} \}, \\
    R_{I,k} &= \{ \mathcal{S} : (N_k,\mathcal{S},\Delta_\Lambda) \text{ is a linked reflexive triple} \}, \\
    R_{\Lambda,k} &= \{ \mathcal{T} : (N_k,\Delta_I,\mathcal{T}) \text{ is a linked reflexive triple} \}.
    \end{align*}
    Let $\sigma_k = \langle E_{I,k} \rangle$. By Lemma \ref{gen implies linked} we have that $\sigma_k \in E_{I,k}$.
    The equivalence $\sigma_k$ contains every element or $E_{I,k}$ and it is easy to see that for any equivalence $\sigma^\prime \subseteq \sigma_k$ the triple $(N_k,\sigma^\prime,\Delta_\Lambda)$ is linked. 
    And so we have the following:
    \begin{equation*}
        E_{I,k} = \{ \mathcal{S} \subseteq \sigma_k : \mathcal{S} \text{ is an equivalence relation} \}.
    \end{equation*}
    Now consider $\langle R_{I,k} \rangle$. Clearly $\sigma_k \in R_{I,k}$ and so $\sigma_k \subseteq \langle R_{I,k} \rangle$.
    Furthermore, for any $\mathcal{S} \in R_{I,k}$ we have $\mathcal{S} \subseteq \langle \mathcal{S} \rangle \subseteq \langle E_{I,k} \rangle = \sigma_k$.
    Hence $\sigma_k = \langle R_{I,k} \rangle$ and:
    \begin{equation*}
        R_{I,k} = \{\mathcal{S} \subseteq \sigma_k : \mathcal{S} \text{ is a reflexive relation} \}.
    \end{equation*}
    Let $I_1 , \dots , I_l$ be the equivalence classes of $\sigma_k$, with sizes $a_1 ,\ldots ,a_l$ respectively.
    Any equivalence relation on $I$ contained in $\sigma_k$ is just a union of equivalence relations on $I_1 , \dots ,I_l$.
    As such there are $B(a_1) \dots B(a_l)$ such relations.
    And similarly any reflexive relation on $X$ contained in $\rho$ is a union of reflexive relations on $I_1 , \dots , I_l$.
    And so there is $2^{a_1^2 - a_1} \dots 2^{a_l^2 - a_l}$ such relations.
    Therefore:
    \begin{equation}
        \label{I fraction form}
        \frac{|E_{I,k}|}{|R_{I,k}|} = \frac{B(a_1) \dots B(a_l)}{2^{a_1^2 - a_1} \dots 2^{a_l^2 - a_l}}
    \end{equation}
    where $a_1 + \dots + a_l = a$.
    Dually we have relations $\tau_k$ on $\Lambda$ such that:
    \begin{align*}
        E_{\Lambda,k} & = \{\mathcal{T} \subseteq \tau_k : \mathcal{T} \text{ is an equivalence relation} \}, \\
        R_{\Lambda,k} & = \{\mathcal{T} \subseteq \tau_k : \mathcal{T} \text{ is a reflexive relation} \}.
    \end{align*}
    Let $\Lambda_1 , \ldots , \Lambda_m$ be the equivalence classes of $\tau_k$, with sizes $b_1 , \ldots , b_m$.
    Then:
    \begin{equation}
        \label{Lambda fraction form}
        \frac{|E_{\Lambda,k}|} {|R_{\Lambda,k}|} = \frac{B(b_1) \dots B(b_m)}{2^{b_1^2 - b_1} \dots 2^{b_m^2 - b_m}}
    \end{equation}
    where $b_1 + \dots + b_m = b$.
    By \eqref{N = G case}, \eqref{I fraction form}, \eqref{Lambda fraction form} and Lemma \ref{Bell number inequality} we have:
    \begin{equation*}
        \frac{|E_{I,1}||E_{\Lambda,1}| }{|R_{I,1}||R_{\Lambda,1}|} = \frac{B(a)B(b)}{2^{a^2 - a}2^{b^2 - b}} \leq \frac{B(a_1) \dots B(a_l)}{2^{a_1^2 - a_1} \dots 2^{a_l^2 - a_l}} \frac{B(b_1) \dots B(b_m)}{2^{b_1^2 - b_1} \dots 2^{b_m^2 - b_m}} = \frac{|E_{I,k}||E_{\Lambda,k}|}{|R_{I,k}||R_{\Lambda,k}|}
    \end{equation*}
    for each $k = 1 \dots , n$.
    Finally by \eqref{DSC coef form for Rees matrix} and Lemma \ref{elementary fraction result} we get:
    \begin{equation*}
        \frac{B(a) B(b)}{2^{a^2 - a}2^{b^2-b}} = \frac{|E_{I,1}||E_{\Lambda,1}| }{|R_{I,1}||R_{\Lambda,1}|} \leq  \frac{|E_{I,1}||E_{\Lambda ,1}| + \dots + |E_{I,n}||E_{\Lambda,n}|}{|R_{I,1}||R_{\Lambda ,1}| + \dots +|R_{I,n}||R_{\Lambda,n}|} =  \chi(S).
    \end{equation*}
    As $a,b >1$ it follows that $S$ is not a group, so by Theorem \ref{motivation theorem} $\chi(S) < 1$.
    This completes the proof \revision{of the first} statement in the theorem.

    For the second part we want to show that for any rational $\alpha \in \left[\frac{B(a)B(b)}{2^{a^2 - a}2^{b^2 - b}},1\right)$, there is a group $G$ and a matrix $P$ with $\chi(\RM) = \alpha$. 
    We saw in the first part that:
    \begin{equation*}
        \frac{|E_{I,k}|}{|R_{I,k}|} = \frac{B(a_1)\dots B(a_l)}{2^{a_1^2 - a_1}\dots2^{a_l^2 - a_l}} 
    \end{equation*}
    where $a_1 , \dots , a_l$ are the sizes of the equivalence classes of $\sigma_k$.
    Suppose that for some $k$ we have $\sigma_k = I \times I$; this is certainly the case when $k = 1$, as then $N_k = G$.
    In this situation we have:
    \begin{equation*}
        \frac{|E_{I,k}|}{|R_{I,k}|} = \frac{B(a)}{2^{a^2 - a}}.
    \end{equation*}
    At the other extreme we might have $\sigma_k = \Delta_I$ for some $k$; this case may or may not happen for a particular Rees matrix semigroup.
    If it does, then:
    \begin{equation*}
        \frac{|E_{I,k}|}{|R_{I,k}|}= \frac{B(1) \dots B(1)}{2^{1^2 - 1} \dots 2^{1^2 - 1}} = 1.
    \end{equation*}
    Analogous results hold for $\frac{|E_{\Lambda,k}|}{|R_{\Lambda,k}|}$.
    
    We shall show that there exists a group $G$ and $\Lambda \times I$ matrix $P$ with entries from $G$ such that:
    \begin{itemize}
        \item $ \frac{|E_{I,k}|}{|R_{I,k}|}$ can only take values $\frac{B(a)}{2^{a^2 -a}}$ or $1$,
        \item $\frac{|E_{\Lambda,k}|}{|R_{\Lambda,k}|}$ can only take values $\frac{B(b)}{2^{b^2 -b}}$ or $1$,
        \item if there are $c$ values of $k$ for which $\frac{|E_{I,k}|}{|R_{I,k}|} = \frac{B(a)}{2^{a^2 -a}}$ then $\frac{|E_{\Lambda,k}|}{|R_{\Lambda,k}|} = \frac{B(b)}{2^{b^2 - b}}$ for the same $c$ values,
        \item if there are $d$ values of $k$ for which $\frac{|E_{I,k}|}{|R_{I,k}|} = 1$ then $\frac{|E_{\Lambda,k}|}{|R_{\Lambda,k}|} = 1$ for the same $d$ values.
    \end{itemize}
    It will then follow from \eqref{DSC coef form for Rees matrix} that we will be able to write the DSC coefficient of this Rees matrix semigroup as:
    \begin{equation*}
        \chi (\RM) = \frac{cB(a)B(b) + d}{c 2^{a^2 -a}2^{b^2 - b}+d}.
    \end{equation*}
    Furthermore we will show that we can obtain any $c \in \N$ and $d \in \N_0$.

    To this end let $p$ be a prime greater than $2^{ab+1} $ and let $G = \Z_{p^k}$, the cyclic group of order $p^k$, for some arbitrary $k \in \N_0$.
    The normal subgroups of $G$ form the chain:
    \begin{equation*}
    \textbf{1} = p^k \Z_{p^k} \leq p^{k-1} \Z_{p^k} \leq \dots \leq p \Z_{p^k} \leq \Z_{p^k} = G.
    \end{equation*}
    Let $r \in [0,k]$ be arbitrary and consider the $ab$ numbers:
    \begin{equation*}
        p^r,2p^r, 2^2 p^r, \dots,2^{ab-2} p^r , 2^{ab-1} p^r.
    \end{equation*}
    Let $P$ be any $\Lambda \times I$ matrix with all of these numbers as its entries.
    We will show that any extract $q_{\lambda \mu i j}$ with $\lambda \neq \mu$ and $i \neq j$ is a generator for $p^r \Z_{p^k}$. 
    If we consider the entries of the matrix as elements of $\Z$ this is equivalent to showing that $p^r$ divides $q_{\lambda \mu i j}$ and that $p^{r+1}$ does not divide $q_{\lambda \mu i j }$. Let:
    \begin{equation*}
        p_{\lambda i} = p^r 2^s, \quad p_{\mu i} = p^r 2^t, \quad p_{\mu j} = p^r 2^u, \quad p_{\lambda j } = p^r 2^v.
    \end{equation*}
    Then by definition $q_{\lambda \mu i j} = p^r(2^s - 2^t + 2^u - 2^v)$, and so $p^r$ divides $q_{\lambda \mu i j}$. 
    Note that:
    \begin{align*}
        | q_{\lambda \mu i j}| & = p^r| 2^s - 2^t +2^u -2^v| \\
        & \leq  p^r( 2^s + 2^t +2^u +2^v) \\
        & \leq p^r(2^{ab - 1} + 2^{ab - 1} + 2^{ab - 1} + 2^{ab - 1}) = p^r2^{ab+1} < p^{r+1}.
    \end{align*}
    As each element of $P$ is distinct we have that $s,t,u$ and $v$ are distinct, without loss of generality let $s$ be the largest of these.
    Then:
    \begin{align*}
        | q_{\lambda \mu i j}| & = p^r|2^s - 2^t + 2^u -2^v|\\
        & \geq p^r(2^s - (2^t +2^u + 2^v)) \\
        & \geq p^r(2^s - (2^{s-1} + 2^{s-2} + 2^{s-3})) \\
        & > p^r(2^s - (2^{s-1} + 2^{s-2} + 2^{s-2})) = 0.
    \end{align*}
    If follows from $0 < |q_{\lambda \mu i j}| < p^{r+1}$, that $p^{r+1}$ does not divide $q_{\lambda \mu i j}$.
    Hence $q_{\lambda \mu i j}$ is a generator for $p^r\Z_{p^k}$ for any distinct $\lambda ,\mu \in \Lambda$ and $i,j \in I$.
    
    If $N$ is a normal subgroup containing $p^r \Z_{p^k}$, then $q_{\lambda \mu i j} \in N$ for any $\lambda,\mu \in \Lambda$ and $i,j \in I$.
    So for any relations $\mathcal{S}$ on $I$ and $\mathcal{T}$ on $\Lambda$ the triple $(N,\mathcal{S},\mathcal{T})$ is linked.
    There are $r+1$ such subgroups.
    Any other normal subgroup $N$ is properly contained in $p^r \Z_{p^k}$ and so does not contain any generator for $p^r\Z_{p^k}$.
    The only extracts $q_{\lambda \mu i j}$ that are not generators for $p^r \Z_{p^k}$ have either $\lambda = \mu$ or $i =j$.
    It follows that $(N,\mathcal{S}, \mathcal{T})$ is a linked triple if and only if $\mathcal{S} = \Delta_I$ and $\mathcal{T} = \Delta_\Lambda$.
    There is $k-r$ such subgroups.
    We finally have that:
    \begin{equation}
        \label{rees matrix construction}
        \chi(S) = \frac{(r+1) B(a)B(b) + k-r}{(r+1)2^{a^2 - a} 2^{b^2 - b} + k - r}.
    \end{equation}
    For any $c \in \N$ and $d \in \N_0$ we can let $k = c+d - 1 \in \N_0$ and $r = c - 1 \leq k$, $\chi(S)$ now becomes:
    \begin{equation*}
        \frac{cB(a)B(b)+ d}{c 2^{a^2 - a} 2^{b^2 - b} + d}
    \end{equation*}
    as claimed.
    To complete the proof let $\beta ,\gamma \in \N$ with $\alpha = \frac{\beta}{\gamma}\in \left[ \frac{B(a)B(b)}{2^{a^2 - a}2^{b^2 - b}} , 1\right) $.
    Letting $c = \gamma - \beta \in \N$ and $d = \beta 2^{a^2 - a} 2^{b^2 - b} - \gamma B(a) B(b) \in \N_0$
    gives $\chi(S) = \alpha$.
\end{proof}

We can now return to a theorem we stated in the introduction about the spectrum of values we can get for the DSC coefficient.
 
\begin{reptheorem}{\ref{Main Theorem}}
    For any $\alpha \in \Q \cap (0,1]$ there is a finite semigroup $S$ with $\chi(S) = \alpha$.
\end{reptheorem}

\begin{proof}
    If $\alpha = 1$, then take $S$ to be any finite group, otherwise $\alpha < 1$. 
    The number of reflexive and symmetric relations on a set of size $n$ is $2^{\frac{n^2 - n}{2}}$.
    As every equivalence relation is reflexive and symmetric, we have that $B(n) \leq 2^{\frac{n^2 - n}{2}}$.
    Dividing by $2^{n^2 - n}$ and letting $n \rightarrow \infty$ gives:
    \begin{equation*}
        \frac{B(n)}{2^{n^2-n}} \leq 2^{-\frac{n^2-n}{2}} \rightarrow 0.
    \end{equation*}
    So for any $\alpha \in \Q \cap (0,1)$ there is $a,b \in \N \setminus \{1 \}$ such that $\alpha \in \left[ \frac{B(a)B(b)}{2^{a^2 - a}2^{b^2 - b}} , 1\right)$.
    By Theorem \ref{main theorem section 3} there is a finite semigroup $S$ with $\chi(S) = \alpha$.
\end{proof}

\section{Conclusion and Further Questions}

The DSC coefficient associates a number to a semigroup, so it natural to ask whether every number can be obtained in this way. 
Theorem \ref{Main Theorem} gives a complete characterization of the numbers we can obtain.
We have done this by analyzing the DSC coefficient for completely simple semigroups, It is natural to try and do the same thing for a different class of semigroups.
Clifford semigroups have a natural description of their congruences so we will conclude with some observations about the DSC coefficient for this class of semigroups.

Clifford semigroups were originally introduced by Clifford \cite{Clifford1941} and were used to give a structural description of regular semigroups where the idempotents are central.
Here we follow the account from \cite[Section 4.2]{Ho95}.
Suppose we have a semilattice $Y$, a collection of pairwise disjoint groups $\{G_\alpha : \alpha \in Y \}$ indexed by $Y$ and for each $\alpha \geq \beta$ a homomorphism  $\phi_{\alpha ,\beta}: G_\alpha \to G_\beta$ such that:
\begin{itemize}
    \item $\phi_{\alpha , \alpha} $ is the identity map on $G_\alpha$,
    \item $\phi_{\alpha,\beta} \phi_{\beta,\gamma} = \phi_{\alpha,\gamma}, \forall \alpha,\beta,\gamma \in Y$ with $\alpha \geq \beta \geq \gamma$.
\end{itemize}
We define the Clifford semigroup $S = \mathcal{S}[Y;G_\alpha ; \phi_{\alpha,\beta}]$ to be the set $\bigcup_{\alpha \in Y}G_{\alpha}$, with multiplication:
\begin{equation*}
    x \cdot y = (x\phi_{\alpha,\alpha \beta})( y \phi_{\beta, \alpha \beta})
\end{equation*}
where $x \in G_{\alpha} , y \in G_{\beta}$ and $x\phi_{\alpha,\alpha \beta} y \phi_{\beta, \alpha \beta}$ is calculated in $G_{\alpha \beta}$.
Let $1_\alpha$ be the identity element of $G_\alpha$ and $E = \{ 1_\alpha : \alpha \in Y \}$ be the collection of idempotents in $S$.
Then $E$ is a subsemigroup of $S$ that is isomorphic to the semilattice $Y$ via the map $1_\alpha \mapsto \alpha$.

Clifford semigroups are a special class of inverse semigroups. 
The congruences on an inverse semigroup were originally described by Scheilblich \cite{Scheiblich_1974} and later developed by Green \cite{MR390093} and Petrich \cite{PETRICH1978231}.
We will use the same language and notation as \cite[Section 5.3]{Ho95}, only adapted to Clifford semigroups.
The congruences on a Clifford semigroup are described by what is called a congruence pair.
For each $\alpha \in Y$ we take a normal subgroup $N_\alpha$ of $G_\alpha$ such that $N_\alpha \phi_{\alpha ,\beta} \subseteq N_\beta$ for all $\alpha \geq \beta$ and let $N = \bigcup_{\alpha \in Y}N_\alpha$, we refer to $N$ as a \textit{kernel}.
Let $\tau$ be a congruence on $E$, which we will call a \textit{trace}.
We say $(N,\tau) $ is \textit{congruence pair} if for $\alpha \geq \beta$ and $x \in G_\alpha$ we have:
\begin{equation}
    \label{lab5}
    x \phi_{\alpha , \beta} \in N \text{ and } (1_\alpha , 1_\beta) \in \tau \implies x \in N.
\end{equation}
Given a congruence pair we can construct the following congruence:
\begin{equation*}
    \rho_{N,\tau } = \{(x,y) : (xx^{-1},yy^{-1}) \in \tau, xy^{-1} \in N \}.
\end{equation*}
Every congruence on a Clifford semigroup can be constructed in this way and each congruence pair gives rise to a unique congruence.
If $S$ is finite, then the number of kernels, $K$, and the number of traces, $|\Cong(Y)|$, are both finite.
Thus $K |\Cong(Y)|$ is an upper bound for $|\Cong(S)|$.

We can deploy a similar construction to obtain a nice family of diagonal subsemigroups on a Clifford semigroup.
We define a diagonal subsemigroup pair in similar way to a congruence pair.
The kernel is defined in the exact same way and the trace $\tau$ is now a diagonal subsemigroup on $E$.
We say the pair $(N ,\tau)$ is a \textit{diagonal subsemigroup pair}, note we do not require $N$ and $\tau$ to satisfy (\ref{lab5}).
We construct $\rho_{N,\tau}$ in the same way as above.

\begin{lemma}
    \label{clifford diagonal subsemigroup family}
    The relation $\rho_{N,\tau}$ is a diagonal subsemigroup.
    Furthermore the map $(N,\tau) \mapsto \rho_{N,\tau}$ from the set of \revision{diagonal subsemigroup} pairs to the set of diagonal subsemigroups is an injection.
\end{lemma}

\begin{proof}
    Clearly $(x x^{-1} , x x^{-1}) \in \tau$ as $\tau$ is reflexive and $x x^{-1} \in N$ as $N$ contains the identity of each $G_\alpha$.
    Let $(x,y),(z,t) \in \rho_{N,\tau}$, with $x \in G_\alpha , y \in G_\beta , z \in G_\gamma$ and $t \in G_\delta$. 
    Then:
    \begin{align*}
        (xz z^{-1} x^{-1}, y t t^{-1} y^{-1}) & = (x 1_\gamma x^{-1} , y 1_\delta y^{-1}) \\
        & = (x x^{-1} \phi_{\alpha, \alpha \gamma} ,yy^{-1} \phi_{\beta,\beta \delta}) \\
        & = (1_{\alpha \gamma} , 1_{\beta \delta}) \\
        & = (1_\alpha, 1_{\beta})(1_{\gamma}, 1_\delta) \\
        & = (xx^{-1},yy^{-1})(zz^{-1},tt^{-1}) \in \tau.
    \end{align*}
    We have $xy^{-1},zt^{-1} \in N$, so $xy^{-1} \in N_{\alpha \beta}$ and $zt^{-1} \in N_{\gamma \delta}$.
    It follows that $y^{-1} x = (y \phi_{\beta , \alpha \beta})^{-1} xy^{-1} y \phi_{\beta , \alpha \beta} \in N_{\alpha \beta}$.
    As $\alpha \beta,\gamma \delta \geq \alpha \beta \gamma \delta$, we see that $y^{-1}x \phi_{\alpha \beta , \alpha \beta \gamma \delta} $ and $ zt^{-1} \phi_{\gamma \delta, \alpha \beta \gamma \delta}$ are both contained in $N_{\alpha \beta \delta \gamma}$.
    Hence $y^{-1} x z t^{-1} \in N_{\alpha \beta \delta \gamma}$.
    Finally conjugating by $y^{-1} \phi_{\beta , \alpha \beta \gamma \delta}$ gives $xz t^{-1}y^{-1} \in N_{\alpha \beta \gamma \delta} \subseteq N$.
    Hence $(xz,yt) \in \rho_{N,\tau}$ and $\rho_{N,\tau}$ is a diagonal subsemigroup.

    Now suppose that $(N,\tau)$ and $(N^\prime,\tau^\prime)$ are diagonal subsemigroup pairs with $\rho_{(N,\tau)} = \rho_{(N^\prime,\tau^\prime)}$.
    It is easy to see that $1_\alpha 1_\beta^{-1} = 1_{\alpha \beta} \in N , N^\prime$, so:
    \begin{equation*}
        (1_\alpha,1_\beta) \in \tau \iff \rho_{(N,\tau)} = \rho_{(N^\prime,\tau^\prime)} \iff (1_\alpha,1_\beta) \in \tau^\prime
    \end{equation*}
    and therefore $\tau = \tau^\prime$.
    To show that $N = N^\prime$ take $x \in G_\alpha$, then $(xx^{-1},1_\alpha 1_\alpha^{-1}) = (1_\alpha,1_\alpha) \in \tau$ and:
    \begin{equation*}
        x \in N \iff (x,1_\alpha) \in \rho_{(N,\tau)} = \rho_{(N^\prime,\tau)} \iff x \in N^\prime. \qedhere
    \end{equation*}
\end{proof}
It is worth noting that this family of diagonal subsemigroups is far from describing all of the diagonal subsemigroups on a Clifford semigroup.
For example let $G$ be a group with identity $e$ and adjoin a new identity $1$ to give $G^1$.
If $H$ is any proper subgroup of $G$ then it is easily seen that $\{(1,1) \} \cup G \times G \cup \{1\} \times H$ is a diagonal subsemigroup of $G^{1}$ that is not of the form $\rho_{N,\tau}$

It follows from Lemma \ref{clifford diagonal subsemigroup family} that there is at least $K |\Diag(Y)|$ diagonal subsemigroups on $S$.
And hence we have:
\begin{equation*}
    \chi(S) = \frac{|\Cong(S)|}{|\Diag(S)|} \leq \frac{K|\Cong(Y)|}{K|\Diag(Y)|} = \chi(Y).
\end{equation*}

\begin{thm}
    \label{Clifford Ratio Result}
    Let $S = \mathcal{S}[Y;G_\alpha ; \phi_{\alpha,\beta}]$ be a finite Clifford semigroup.
    Then $0 < \chi(S) \leq \chi(Y)$. \qed
\end{thm}

The example of a group with an identity adjoined discussed above shows that the second equality can be strict.

It is perhaps interesting to compare Theorem \ref{Clifford Ratio Result} with Theorem \ref{main theorem section 3} for Rees matrix semigroups.
Both theorems exhibit a subinterval of $(0,1]$ which necessarily contains $\chi(S)$. 
However, in Theorem \ref{main theorem section 3}, as the index sets $I,\Lambda$ become larger this interval rapidly converges towards $(0,1]$.
We do not know what the possible values of $\chi(Y)$ are for a semilattice $Y$, but examples seem to suggest that it converges to $0$ as $|Y|$ increases, yielding very narrow intervals for $\chi(S)$. 

We also note that distinguished homomorphic images of $S$ play a prominent role in both inequalities.
Indeed the semilattice $Y$ is a natural homomorphic image of $S$ in the Clifford case, and $\chi(Y)$ is an upper bound for $\chi(S)$.
By way of contrast, when $S= \RM$ is a Rees matrix semigroup with $|I| = a , |\Lambda| = b$, the $I \times \Lambda$ rectangular band is a homomorphic image of $S$.
As we observed earlier its DSC coefficient is $\frac{B(a)B(b)}{2^{a^2 - a}2^{b^2 - b}}$ which is precisely the lower bound for $\chi(S)$ as obtained in Theorem \ref{main theorem section 3}.


\begin{thebibliography}{1}

\bibitem{BARBER_2025}
\revision{
C.~Barber and N.~Ru{\v{s}}kuc.
\newblock Semigroup congruences and subsemigroups of the direct square.
\newblock {\em Bull. Aust. Math. Soc.}, 113(1):123--134, 2026.}

\bibitem{Clifford1941}
A.~H. Clifford.
\newblock Semigroups admitting relative inverses.
\newblock {\em Annals of Mathematics}, 42(4):1037--1049, 1941.

\bibitem{MR390093}
D.~G. Green.
\newblock The lattice of congruences on an inverse semigroup.
\newblock {\em Pacific J. Math.}, 57(1):141--152, 1975.

\bibitem{Ho95}
J.~M. Howie.
\newblock {\em Fundamentals of semigroup theory}, volume~12 of {\em London
  Mathematical Society Monographs. New Series}.
\newblock The Clarendon Press, Oxford University Press, New York, 1995.
\newblock Oxford Science Publications.

\bibitem{kapp1969completely}
K.~Kapp and H.~Schneider.
\newblock {\em Completely O-simple Semigroups: An Abstract Treatment of the
  Lattice of Congruences}.
\newblock 99-0108427-X. W. A. Benjamin, 1969.

\bibitem{PETRICH1978231}
M.~Petrich.
\newblock Congruences on inverse semigroups.
\newblock {\em Journal of Algebra}, 55(2):231--256, 1978.

\bibitem{Preston_1961}
G.~B. Preston.
\newblock Congruences on completely 0-simple semigroups.
\newblock {\em Proceedings of the London Mathematical Society},
  s3-11(1):557--576, 1961.

\bibitem{Scheiblich_1974}
H.~E. Scheiblich.
\newblock Kernels of inverse semigroup homomorphisms.
\newblock {\em Journal of the Australian Mathematical Society},
  18(3):289–292, 1974.

\bibitem{Su28}
A.~Suschkewitsch.
\newblock \"uber die endlichen {G}ruppen ohne das {G}esetz der eindeutigen
  {U}mkehrbarkeit.
\newblock {\em Math. Ann.}, 99(1):30--50, 1928.

\end{thebibliography}
\end{document}